\documentclass[12pt]{article}
\usepackage{amsmath, amssymb, amscd, amsthm, amsfonts, xfrac, enumerate}
\usepackage[titletoc,title]{appendix}
\numberwithin{equation}{section}
\usepackage{graphicx}
\usepackage{hyperref}
\usepackage{authblk}
\usepackage{url}
\hypersetup{backref, colorlinks=true}
\newtheorem{theorem}{Theorem}[section]
\newtheorem{remark}[theorem]{Remark}

\newtheorem{exam}[theorem]{Example}
\newtheorem{lemma}[theorem]{Lemma}

\title{\textbf{Positivity of arbitrary-order P-recursive sequences with a unique dominant root}}
\author{Zhongjie Li}
\affil{School of Mathematics and KL-AAGDM \break Tianjin University \break Tianjin 300350, China
\break\texttt{lizhongjie@tju.edu.cn}}
\date{}

\begin{document}
\maketitle

\begin{abstract}
We establish a sufficient condition for the ultimate positivity of P-recursive sequences of arbitrary order with a unique dominant root. By additionally verifying finitely many initial terms, the positivity can also be resolved. As an application, we provide several examples of P-recursive sequences of order greater than two.
\end{abstract}


\section{Introduction}

Let $\{a_n\}_{n \ge 1}$ be a \emph{P-recursive sequence}. Recall that a P-recursive sequence of order $d$ satisfies a recurrence relation of the form
\[
p_0(n) a_n + p_1(n) a_{n+1} + \cdots + p_d(n) a_{n+d} = 0, \quad \forall\, n \ge 1,
\]
where $p_{i}(n)$ are polynomials in $n$. The positivity of P-recursive sequences has long been a topic of significant interest. Ouaknine and Worrell \cite{Ouaknine2014Positivity} introduced two fundamental \emph{decision problems}: the \emph{Positivity Problem} and the \emph{Ultimate Positivity Problem}. In this paper, we focus on both problems in the setting of P-recursive sequences. The Positivity Problem asks whether all terms of  a given P-recursive sequence $\{a_n\}_{n \ge{0}}$ are positive, whereas the Ultimate Positivity Problem asks whether all but finitely many terms of such a sequence are positive. 

In recent years, the study of positivity for different sequences has attracted considerable attention. Halava {\it{et al.}} \cite{Halava2006Positivity} established the decidability of the Positivity Problem for second order recurrent sequences with constant coefficients. Kenison {\it{et al.}} \cite{Kenison2021On} demonstrated that the Positivity Problem for second-order P-recursive sequences is decidable under certain conditions. Xia and Yao \cite{Xia2011The}, as well as Pei {\it{et al.}} \cite{Pei2023Positivity}, gave various sufficient conditions for three-term recurrences, which have been applied to prove the positivity for classical sequences such as the central Delannoy numbers, the Catalan-Larcombe-French numbers, and the Ap\'{e}ry numbers.

Ouaknine and Worrell \cite{Ouaknine2014On} established that both the Positivity Problem and Ultimate Positivity Problem are decidable for linear recurrence sequences of order 5 or less. They further proved in \cite{Ouaknine2014Positivity} that the Positivity Problem for simple linear recurrence sequences of order 9 or less is decidable. Moreover, in \cite{Ouaknine2014Ultimate}, they showed the decidability of the Ultimate Positivity Problem for simple linear recurrence sequences of arbitrary order. In addition, Kauers and Pillwein \cite{Kauers2010When}, as well as Ibrahim and Salvy \cite{Ibrahim2024Positivity}, respectively proposed algorithmic frameworks for verifying the positivity of P-finite sequences. For further results on the positivity of sequences, see \cite{Bell2006On, Laohakosol2009Positivity, Liu2010Positivity, Stanley2000Positivity, Straub2008Positivity}.

In this paper, we extend the method of Xia and Yao \cite{Xia2011The} from three-term recurrences to P-recursive sequences of arbitrary order. While the positivity of low-order recurrences has been extensively studied, the arbitrary-order case is substantially more complicated due to the simultaneous presence of multiple terms with mixed signs. To address this difficulty, we establish an explicit sufficient condition for positivity based on suitable bounds for the ratio sequence ${a_n}/{a_{n-1}}$. 

We consider a P-recursive sequence $\{a_n\}_{n\ge 0}$ of order $d$ defined by the following recurrence relation
\begin{align}\label{t(1.1)}
a_n = \sum_{j=1}^d \frac{P_{1j}(n)}{P_{2j}(n)} a_{n-j},
\end{align}
where $P_{ij}(n) = \sum_{l=0}^k a_{ij}^{(l)} n^l$ with $a_{ij}^{(k)}\neq 0$ for $i = 1,2$ and $j = 1, 2, \dots, d$. 

Without loss of generality, we assume that the leading coefficients $a_{2j}^{(k)}$ of all denominator polynomials $P_{2j}(n)$ are strictly positive for $1 \le j \le d$. Indeed, if this is not the case, we can multiply both $P_{1j}(n)$ and $P_{2j}(n)$ by $-1$.

To ensure that the leading coefficients $a_{1j}^{(k)}$ of the numerators $P_{1j}(n)$ are positive, we introduce the sign function 
\(\omega_j = \mathrm{sgn}(a_{1j}^{(k)})\) and define
\[
Q_{1j}(n) = \omega_j P_{1j}(n) = \sum_{l=0}^k b_{1j}^{(l)} n^l,
\]
so that the leading coefficients satisfy \(b_{1j}^{(k)} = |a_{1j}^{(k)}| > 0\). 
With this normalization, the recurrence relation can be rewritten as
\begin{align}\label{t(1.2)}
a_n = \sum_{j=1}^d \omega_j \frac{Q_{1j}(n)}{P_{2j}(n)} a_{n-j}.
\end{align}

Following \cite{Mcintosh1993Recurrences}, we associate with the recurrence (\ref{t(1.2)}) the characteristic polynomial
\begin{align}\label{t(1.3)}
p(t) = t^{d} - \sum_{j=1}^{d} \left(\omega_{j} \frac{b_{1j}^{(k)}}{a_{2j}^{(k)}} \right) t^{d-j}. 
\end{align}
Let $\mu_1, \dots, \mu_d$ denote the roots of the characteristic polynomial (\ref{t(1.3)}). According to Definition 2.2 in \cite{Ibrahim2024Positivity}, these roots can be ordered by decreasing modulus as
\[|\mu_1| = \dots = |\mu_s| > |\mu_{s+1}| \ge \dots \ge |\mu_d|.\]
Then the roots $\mu_1, \dots, \mu_s$ are called the \emph{dominant roots} of (\ref{t(1.3)}). Furthermore, we say that a root is \emph{simple} if it is a simple root of the characteristic polynomial. In this paper, we consider only those instances where there exists a unique positive real dominant root, denoted by $\mu$. Under this assumption, we can deduce suitable bounds for the ratio sequence.

\indent This paper is organized as follows. In Section \ref{s2}, we introduce some necessary notation and establish a sufficient condition for the ultimate positivity and positivity of the sequence $a_n$. In Section \ref{s3}, we apply this sufficient condition to several examples illustrating the positivity of P-recursive sequences of order greater than two.

\section{The positivity of the sequence $\{a_n\}_{n\ge 0}$}\label{s2}
In this section, we consider the Ultimate Positivity Problem and the Positivity Problem for the P-recursive sequences. Our approach is based on establishing suitable bounds for the ratio $a_n/a_{n-1}$.

We begin by introducing some notation that will be used throughout this section.

To simplify the analysis, we adopt the operator $\mathcal L$ from \cite{Xia2011The}. Let $h(n) = \sum_{i=0}^{t} e_i n^i$ be a polynomial with $e_t \neq 0$.  We define the operator $\mathcal L$ on $h(n)$ by
\[
\mathcal{L}(h(n)) = \sum_{\substack{0 \le i \le t-1 \\ e_i < 0}} |e_i|.
\]
Let $p$ and $q$ be real numbers satisfying $0 < p < \mu < q$, and assume that the following two constants are both positive,
\begin{align}
p_0 & = \sum_{j=1}^d \frac{\omega_j b_{1j}^{(k)} A_j}{l_j^{j-1}} - p \prod_{j=1}^d a_{2j}^{(k)} > 0,\label{t(2.5)}\\
q_0 & = q \prod_{j=1}^d a_{2j}^{(k)} - \sum_{j=1}^d \frac{\omega_j b_{1j}^{(k)} A_j}{h_j^{j-1}} > 0,\label{t(2.6)}
\end{align}
where
\[
l_j = 
\begin{cases} 
p, & \omega_j = -1 \\ 
q, & \omega_j = 1 
\end{cases}, \quad 
h_j = 
\begin{cases} 
q, & \omega_j = -1 \\ 
p, & \omega_j = 1 
\end{cases},
\]
and $A_j = \prod_{i \neq j} a_{2i}^{(k)}$ denotes the product of the leading coefficients of all denominators except the $j$-th one. 

With these constants, we define two polynomials $f(n)$ and $g(n)$ by
\begin{align}
f(n) &= \sum_{j=1}^d \frac{\omega_j Q_{1j}(n+1) A_j(n+1)}{l_j^{j-1}} - p \prod_{j=1}^d P_{2j}(n+1),\label{t(2.7)}\\
g(n) &= q \prod_{j=1}^d P_{2j}(n+1) - \sum_{j=1}^d \frac{\omega_j Q_{1j}(n+1) A_j(n+1)}{h_j^{j-1}} ,\label{t(2.8)}
\end{align}
where $A_j(n) = \prod_{i \neq j} P_{2i}(n)$ is the product of all denominators except the $j$-th one.

We now define a threshold by
\begin{align}\label{t(2.9)}
r = \max_{1\le j \le d} \left\{ \left\lfloor \frac{\mathcal{L}(f(n))}{p_0} \right\rfloor, \left\lfloor \frac{\mathcal{L}(g(n))}{q_0} \right\rfloor, \left\lfloor \frac{\mathcal{L}(Q_{1j}(n))}{b_{1j}^{(k)}} \right\rfloor, \left\lfloor \frac{\mathcal{L}(P_{2j}(n))}{a_{2j}^{(k)}} \right\rfloor \right\} + 1.
\end{align}
The choice of $r$ ensures the following lemma.

\begin{lemma}\label{l2.1}
The polynomials $f(n)$, $g(n)$, and $Q_{1j}(n)$, $P_{2j}(n)$ for each $1\le j\le d$, are positive for all $n\ge r$.
\end{lemma}
\begin{proof}
We first consider $f(n)$. Its leading coefficient is 
\[\sum_{j=1}^d \frac{\omega_j b_{1j}^{(k)} A_j}{l_j^{j-1}} - p \prod_{j=1}^d a_{2j}^{(k)}=p_{0} > 0.\]

Write
\[
f(n) = p_0 n^{dk} + \sum_{0\le i\le dk-1} F_i n^i,
\]
where $F_i$ denotes the coefficients of $n^i$ in $f(n)$.

By the definition of $r$, we have
\begin{align*}
f(n) & \ge p_0 n^{dk} + \sum_{\substack{0\le i\le dk-1 \\ F_i < 0}} F_i n^i\\
& \ge p_0 n^{dk}-\mathcal{L}(f(n)) n^{dk-1}\\
&= n^{dk-1} \bigl(p_0 n - \mathcal{L}(f(n))\bigr)\\
& > 0.
\end{align*}

The same argument applies to $g(n)$, $Q_{1j}(n)$ and $P_{2j}(n)$ (for each $1\le j\le d$), completing the proof.
\end{proof}

Based on Lemma \ref{l2.1}, we next derive sufficient conditions to guarantee the positivity and ultimate positivity of the P-recursive sequence $\{a_n\}_{n\ge 0}$.

\begin{theorem}\label{t2.2}
Let $a_n$ be a P-recursive sequence defined by the recurrence (\ref{t(1.2)}). Suppose that
\[\lim_{n \to \infty} \frac{a_n}{a_{n-1}} = \mu.\]
Then there exists an integer $u \ge r$ such that for all $n\ge u$, the following inequality holds
\begin{align}\label{t(2.1)}
p < \frac{a_n}{a_{n-1}} < q.
\end{align}
Furthermore, for every such $u$, if $a_{u} > 0$, then the sequence $\{a_n\}_{n \ge u}$ is positive. In particular, if $a_n > 0$ for $0\le n\le u$, the sequence $\{a_n\}_{n \ge 0}$ is positive.
\end{theorem}

\begin{proof}
We prove the inequality (\ref{t(2.1)}) by induction on $n$. 

We call an integer $u \ge r$ \emph{admissible} if
\[
p < \frac{a_n}{a_{n-1}} < q \quad \text{for all } n = u, u+1, \dots, u+d-1.
\]
Since
\[
0 < p < \lim_{n \to \infty} \frac{a_n}{a_{n-1}} = \mu < q,
\]
there exists an integer $N$ such that the inequality (\ref{t(2.1)}) holds for all $n \ge N$. 
Hence every $u \ge N$ is admissible, so admissible integers $u$ exist. In addition, there may also exist admissible integers with $u < N$.

Fix an arbitrary admissible integer $u \ge r$. Then the inequality (\ref{t(2.1)}) holds for all $n = u, u+1, \dots, u+d-1$. Now assume that the inequality (\ref{t(2.1)}) holds for $u \le n \le m$, where $m\ge u+d-1$. It follows that
\[p<\frac{a_m}{a_{m-1}}<q.\]
This implies that for any $1\le k\le m-u+1$, we have 
\[p^k <\frac{a_m}{a_{m-k}} < q^k,\]
which can be rewritten as
\begin{align}\label{t(2.2)}
\frac{1}{q^k} < \frac{a_{m-k}}{a_{m}} < \frac{1}{p^k}.
\end{align}

We now prove the case for $n = m+1$. First, we consider the left-hand side of the inequality (\ref{t(2.1)}). Substituting the recurrence relation (\ref{t(1.2)}) satisfied by $a_n$, we obtain
\begin{align}
& \frac{a_{m+1}}{a_m} - p \nonumber\\
= & \sum_{j=1}^d \omega_j \frac{Q_{1j}(m+1)}{P_{2j}(m+1)}  \frac{a_{m-j+1}}{a_m} - p \nonumber\\ 
= & \frac{\sum_{j=1}^d \omega_j \frac{a_{m-j+1}}{a_{m}} Q_{1j}(m+1) A_j(m+1)  - p \prod_{j=1}^d P_{2j}(m+1)}{\prod_{j=1}^d P_{2j}(m+1)}.\label{t(2.3)}
\end{align}
Based on the definition of $l_j$, we have
\begin{align}\label{t(2.4)}
\omega_j \frac{a_{m-j+1}}{a_m} > 
\begin{cases} 
\frac{1}{q^{j-1}}, & \omega_j = 1 \\ 
-\frac{1}{p^{j-1}}, & \omega_j = -1 
\end{cases}
=\frac{\omega_j}{l_{j}^{j-1}} \quad \text{for } 2\le j\le d.
\end{align}
Combining (\ref{t(2.3)}) and (\ref{t(2.4)}), we derive
\begin{align*}
\frac{a_{m+1}}{a_m} - p > \frac{f(m)}{\prod_{j=1}^d P_{2j}(m+1)}.
\end{align*}
By Lemma \ref{l2.1}, we know that $f(m) > 0$ and $P_{2j}(m+1) > 0$ for all $m \ge u \ge r$. Therefore, 
\[\frac{a_{m+1}}{a_{m}} > p .\]

Similarly, we can deduce that
\[
q - \frac{a_{m+1}}{a_m} > \frac{g(m)}{\prod_{j=1}^d P_{2j}(m+1)} > 0,
\]
which implies
\[
q > \frac{a_{m+1}}{a_m}.
\]

Hence,
\[
p < \frac{a_{m+1}}{a_m} < q.
\]

Therefore, by induction, the inequality (\ref{t(2.1)}) holds for all $n \ge u$. 

Since $\frac{a_n}{a_{n-1}} > p > 0$, the sequence $a_n$ has constant sign for all $n \ge u$. In particular, if $a_u > 0$, then $a_n > 0$ for all $n \ge u$. 

If moreover $a_n > 0$ for $0 \le n \le u$, then the sequence $a_n$ is positive for all $n \ge 0$.

\end{proof}

\begin{remark}\label{r2.3}
\upshape
Based on Theorem \ref{t2.2}, the verification of positivity reduces to checking the following conditions:
\begin{enumerate}[(1)]
    \item $p < \frac{a_n}{a_{n-1}} < q$ for $n = u, u+1, \dots, u+d-1$;
    \item $a_u > 0$;
    \item $a_n > 0$ for $0 \le n \le u$.
\end{enumerate}
Conditions (1) and (2) ensure ultimate positivity, while conditions (1) and (3) ensure positivity of the entire sequence.

To minimize the verification effort, it is sufficient to determine the smallest admissible integer $u$. Since Theorem \ref{t2.2} guarantees the existence of an admissible integer $u \ge r$, we can start from $u = r$ and check condition (1). If it is satisfied, then $u$ is admissible; otherwise, increase $u$ by one and repeat. This procedure terminates after finitely many steps and yields the minimal admissible $u$.
\end{remark}

\section{Applications}\label{s3}
In this section, we provide several examples to demonstrate the effectiveness of Theorem \ref{t2.2} in handling higher-order recurrences for which previous positivity criteria are difficult to use.

\begin{exam}\label{e3.1}
Let $f_n$ denote the Franel number of order 5. It is known from \cite{Cusick1989Recurrences} that $f_n$ is a P-recursive sequence of order $d=3$ satisfying the recurrence relation, 
\begin{equation}\label{t(3.1)}
f_n = \frac{Q_{11}(n)}{P_{21}(n)}f_{n-1} + \frac{Q_{12}(n)}{P_{22}(n)}f_{n-2} - \frac{Q_{13}(n)}{P_{23}(n)}f_{n-3},
\end{equation}
with initial values $f_0 = 1$, $f_1 = 2$ and $f_2 = 34$, where the polynomial coefficients are given by
\begin{align*}
P_{21}(n) & = P_{22}(n) = P_{23}(n) =  n^4(55n^2 - 187n + 160), \\
Q_{11}(n) & = 1155n^6 - 6237n^5 + 13128n^4 - 13957n^3 + 8193n^2 - 2562n + 336, \\
Q_{12}(n) & = 19415n^6 - 143671n^5 + 434583n^4 - 686633n^3 + 596930n^2 \\
&\quad - 270704n + 50176, \\
Q_{13}(n) & = 32(n - 2)^4(55n^2 - 77n + 28).
\end{align*}
We aim to prove that $f_n > 0$ for all $n \ge 0$.
\end{exam}
To study the positivity of $f_n$, we consider the characteristic polynomial $p(t)$ associated with the recurrence relation (\ref{t(3.1)}),
\[
p(t) = t^3 - 21t^2 - 353t + 32.
\]

It follows from a direct computation that $t = 32$ is the unique positive dominant root of $p(t)$.
Hence, we have 
\[
\mu = 32.
\]

We choose $p = 30$ and $q =33$ such that $0 < p < \mu = 32 < q$. Substituting the values into  (\ref{t(2.5)}) and (\ref{t(2.6)}), we obtain
\[
p_0 = \frac{2487760}{9},\quad q_0 = \frac{786775}{18}.
\]

Next, we compute $f(n)$ and $g(n)$ using their definition in (\ref{t(2.7)}) and (\ref{t(2.8)}). They are both degree-$18$ polynomials with positive leading coefficients,
\begin{align*}
f(n) & = \frac{2487760}{9}n^{18} - \frac{35745094}{3}n^{17} + \cdots - \frac{505347584}{825},\\
g(n) & = \frac{786775}{18}n^{18} + \frac{90780415}{6}n^{17} + \cdots + \frac{1230826688}{1815}.
\end{align*}

By the definition of (\ref{t(2.9)}), we obtain
\[ r = 27099.\]

Using {\tt{Mathematica}}, we verify that
\begin{align*}
30 < \frac{f_{27099}}{f_{27098}} < 33,\quad
30 < \frac{f_{27100}}{f_{27099}} < 33,\quad
30 < \frac{f_{27101}}{f_{27100}} < 33.
\end{align*}
Therefore, based on condition (1) in Remark \ref{r2.3}, we know that $u = r = 27099$ is the minimal admissible $u$. 

Since $f_{27099} > 0$, it follows from Theorem \ref{t2.2} that 
\[f_n > 0 \quad \text{for all } n \ge 27099.\]

Together with the verification that $f_n > 0$ for $0 \le n \le 27098$, we conclude that 
\[f_n > 0 \quad \text{for all } n \ge 0.\]

\begin{exam}
Let $b_n$ denote the diagonal coefficients of the Gillis-Reznick-Zeilberger rational function for $r=4$. It was shown in \cite{Pillwein2019On} that the sequence $b_n$ is P-recursive of order $d = 4$, satisfying a recurrence relation of the form
\begin{equation}\label{t(3.2)}
b_n = \frac{Q_{11}(n)}{P_{21}(n)}b_{n-1} - \frac{Q_{12}(n)}{P_{22}(n)}b_{n-2} - \frac{Q_{13}(n)}{P_{23}(n)}b_{n-3} - \frac{Q_{14}(n)}{P_{24}(n)}b_{n-4},
\end{equation}
for $n\ge4$, with initial values
\[
b_0=1,\quad
b_1=0,\quad
b_2=216,\quad
b_3=18816,
\]
and the polynomial coefficients are given by
\begin{align*}
P_{21}(n) & = P_{22}(n) = P_{23}(n) = P_{24}(n) = n^3(2n-3)(4n-9)(4n-5), \\
Q_{11}(n) & = 8(n-1)(4n-9)(4n-3)(40n^3-100n^2+73n-12), \\
Q_{12}(n) & =  576(192n^6-1536n^5+4748n^4-7050n^3+5065n^2-1563n+171),\\
Q_{13}(n) & = 13824(4n-1)(32n^5-296n^4+1040n^3-1689n^2+1209n-279),\\
Q_{14}(n) & = 331776(2n-1)(4n-5)(4n-1)(n-3)^3.
\end{align*}
We now show that $b_n>0$ for all $n\ge 2$.

\end{exam}

The characteristic polynomial associated with recurrence (\ref{t(3.2)}) is
\[
p(t)=t^4-160t^3+3456t^2+55296t+331776.
\]

It follows that the polynomial $p(t)$ has the unique dominant positive root
\[
\mu=129.99.
\]

Choose $p=64$ and $q=226$ so that $0<p<\mu = 129.99 <q$. A direct computation gives
\[p_0 = 28557312 > 0,\quad q_0 = \frac{124675285843968}{1442897} >0.\]

Using the definitions in Section \ref{s2}, we compute the corresponding polynomials $f(n)$ and $g(n)$,
\begin{align*}
f(n) & = 28557312n^{24} + 104103936n^{23} +\cdots + \frac{31375}{2},\\
g(n) & = \frac{124675285843968}{1442897}n^{24} + \frac{784776214609920}{1442897}n^{23} + \cdots + \frac{188783701250}{1442897}.
\end{align*}

Hence, $r=1148$.

Next, we check that
\[
64<\frac{b_n}{b_{n-1}}<226, \quad \text {for } 1148\le n\le 1151.
\]
Therefore, $u=1148$ is admissible.

Since $b_{1148} >0$, Theorem \ref{t2.2} implies that
\[
b_n>0
\quad\text{for all } n\ge1148.
\]

Combining this with $b_n>0$ for $2\le n\le 1147$, we conclude that
\[
b_n>0
\quad\text{for all } n\ge2.
\]

\begin{exam}\label{e3.3}
Let $c_n$ denote the number of hill-free Dyck paths of semilength $n$
avoiding the pattern $UUDD$, where $U=(1,1)$ and $D=(1,-1)$.
This sequence corresponds to OEIS A105641 and belongs to
the family of Fine-number refinements studied in
\cite{Deutsch2001Survey}.
The sequence $\{c_n\}_{n\ge 2}$ is P-recursive of order $6$ and satisfies
the recurrence
\begin{align}\label{t(3.3)}
c_n
&= \frac{7n-5}{2(n+1)}\,c_{n-1}
 - \frac{n-5}{2(n+1)}\,c_{n-2}
 + \frac{2(n+1)}{2(n+1)}\,c_{n-3}
 - \frac{2(2n-7)}{2(n+1)}\,c_{n-4} \nonumber\\
&\quad
 - \frac{n-5}{2(n+1)}\,c_{n-5}
 - \frac{n-5}{2(n+1)}\,c_{n-6},
\end{align}
for $n\ge 8$, with initial values
\[
c_2 = 0,\quad
c_3 = 1,\quad
c_4 = 2,\quad
c_5 = 5,\quad
c_6 = 14,\quad
c_7 = 39.
\]
We now apply Theorem~\ref{t2.2} to prove that
\[
c_n > 0
\qquad \text{for all } n\ge 3.
\]
\end{exam}

The characteristic polynomial associated with (\ref{t(3.3)}) is
\[
p(t)=t^6-\frac{7}{2}t^5+\frac{1}{2}t^4-t^3+2t^2+\frac{1}{2}t+\frac{1}{2},
\]
whose unique dominant positive root is
\[
\mu=3.38298.
\]

Choose $p=3$ and $q=\frac{7}{2}$ such that $0<p<\mu = 3.38298 <q$. A direct verification shows that
\[p_0 = \frac{253504}{11907} > 0,\quad q_0 = \frac{800384}{151263} >0.\]

Moreover,
\begin{align*}
f(n) & = \frac{253504}{11907}n^{6} - \frac{158848}{3969}n^5 - \cdots - \frac{96370688}{11907},\\
g(n) & = \frac{800384}{151263}n^{6} + \frac{19036928}{50421}n^5 + \cdots + \frac{1571422208}{151263}.
\end{align*}

Hence, $r=2645$.

Next, we check that
\[
3<\frac{c_n}{c_{n-1}}<\frac{7}{2}, \quad \text {for } 2645\le n\le 2650.
\]
Therefore, $u=r =2645$ is admissible.

Since $c_{2645} >0$, Theorem \ref{t2.2} yields
\[
c_n>0
\quad\text{for all } n\ge2645.
\]

Combining this with the positivity of $c_n$ for $3 \le n \le 2644$, we conclude that
\[
c_n>0
\quad\text{for all } n\ge 3.
\]

\begin{remark}
\upshape
The explicit symbolic computations are omitted from the paper for readability. They are available in the accompanying Mathematica notebook.
\end{remark}

\section*{Acknowledgments}
The author would like to thank professor Manuel Kauers for his valuable suggestions. This work was supported by the \emph{China Scholarship Council} (No. 202506250062). Part of this work was completed during her visit at Johannes Kepler University in 2025--2026.

\end{document}